\documentclass{amsart}
\usepackage{amsfonts}
\usepackage{amsmath,amssymb}
\usepackage{amsthm}
\usepackage{amscd}
\usepackage{graphics}
\usepackage{graphicx}
\theoremstyle{definition}{
	\newtheorem{Def}{{\rm Definition}}
	\newtheorem{Ex}{{\rm Example}}
	\newtheorem{Rem}{{\rm Remark}}
	
}
\theoremstyle{plain}
{
	
	\newtheorem{Prop}{Proposition}
	\newtheorem{Thm}{Theorem}
	\newtheorem{MainThm}{Main Theorem}
	
	\newtheorem{Lem}{Lemma}
	
}

\begin{document}
	\title[Round fold maps on $3$-dimensional manifolds and cohomology rings]{Round fold maps on $3$-dimensional manifolds and their integral and rational cohomology rings}
	\author{Naoki Kitazawa}
	\address{Institute of Mathematics for Industry, Kyushu University, 744 Motooka, Nishi-ku Fukuoka 819-0395, Japan\\
		TEL (Office): +81-92-802-4402 \\
		FAX (Office): +81-92-802-4405 \\
	}
	\email{n-kitazawa@imi.kyushu-u.ac.jp}
	\urladdr{https://naokikitazawa.github.io/NaokiKitazawa.html}
	
	\begin{abstract}
		{\it Fold} maps are smooth maps at each singular point of which it is represented as the product map of a Morse function and the identity map.
		{\it Round} fold maps are, in short, such maps the sets of all singular points of which are embedded concentrically. They are, as Morse functions, important in understanding the topologies and the differentiable structures in geometric ways. 
		
		In the present paper, we study cohomology rings of $3$-dimensional manifolds admitting round fold maps into the plane and see that difference of the coefficient rings and topological types of round fold maps are closely related. This is an explicit precise new study on our previous study, showing that a $3$-dimensional closed and orientable manifold is a so-called {\it graph manifold}, or a manifold obtained by gluing so-called circle bundles over surfaces along tori, if and only if it admits a round fold map into the plane. We also show another exposition on classifications of graph manifolds admitting such maps whose topological types are of a certain simplest class.


		
	\end{abstract}
	
	
	\maketitle
	\section{Introduction.}
	\label{sec:1}
	
{\it Fold} maps are higher dimensional versions of Morse functions. They are locally represented as the projections or the products of Morse functions and the identity maps on open sets in Euclidean spaces. 
{\it Round} fold maps have been introduced by the author in \cite{kitazawa0.1, kitazawa0.2, kitazawa0.3}, followed by \cite{kitazawa0.4,kitazawa0.5,kitazawa0.6} for example. Studies such as \cite{kitazawasaeki1,kitazawasaeki2} have been presented recently. These maps have been fundamental and strong tools in understanding the topologies and the differentiable structures of the manifolds via geometric and constructive ways or combinatorial ways. They are also strong tools in knowing not only (co)homology groups of these manifolds, but also more precise information such as fundamental groups, cohomology rings and differentiable structures.

Our paper is on round fold maps on $3$-dimensional closed and orientable manifolds. It has been shown that such a manifold admits a round fold map into the plane if and only if it is a so-called graph manifold. A {\it graph manifold} is, in short, a manifold obtained by gluing so-called {\it circle bundles} over surfaces or bundles over surfaces whose fibers are circles along tori.
A round fold map of a certain simplest class is defined as {\it directed}. Graph manifolds admitting such maps are characterized in terms of \cite{neumann}, or graphs with several labels used here and simpler graphs. 
In short the class of such manifolds are ones whose {\it normal forms}, uniquely defined graphs with the labels, are trees.
$3$-dimensional spheres, circle bundles over spheres, Lens spaces and Seifert manifolds over spheres are of such a simplest class of $3$-dimensional manifolds. 

Our Main Theorems are as follows. In our paper, elementary algebraic topology is fundamental. For related fundamental terminologies, notions, and notation, see \cite{hatcher1} for example. 

\begin{MainThm}
	\label{mthm:1}
If a graph manifold $M$ admits a directed round fold map into the plane ${\mathbb{R}}^2$, then for any ordered pair of elements of the 1st integral cohomology group $H^1(M;\mathbb{Z})$, the cup product is the zero element. 
	\end{MainThm}
This has been shown in \cite{kitazawasaeki1} in the case where the coefficient ring is the field $\mathbb{Q} \supset \mathbb{Z}$ of all rational numbers. This previous result is shown by applying \cite{doighorn} for example after our characterization of $3$-dimensional closed and orientable manifolds admitting directed round fold maps into the plane ${\mathbb{R}}^2$. See Theorems \ref{thm:4} and \ref{thm:5}, presented later, for example. This can be also regarded as an extension of a result in \cite{doighorn}, where Theorem \ref{thm:4} is applied.
\begin{MainThm}
	\label{mthm:2}
	There exists a family $\{M_j\}$ of countably many $3$-dimensonal closed, connected and orientable manifolds admitting no directed round fold maps into the plane ${\mathbb{R}}^2$ enjoying the following properties. 
	\begin{enumerate}
		\item Distinct manifolds in the family are mutually non-homemorphic.
		\item For each manifold $M_j$ in the family, there exists a $3$-dimensional closed, connected and orientable manifold $M_{j,0}$ enjoying the following properties.
\begin{enumerate}
\item The integral homology groups of the original manifold $M_j$ and the manifold $M_{j,0}$ are mutually isomorphic.
\item The rational cohomology rings of the original manifold $M_j$ and the manifold $M_{j,0}$ are mutually isomorphic.
\item The integral cohomology rings of the original manifold $M_j$ and the manifold $M_{j,0}$ are mutually non-isomorphic.
\item $M_{j,0}$ admits directed round fold maps into the plane ${\mathbb{R}}^2$.
\end{enumerate} 
		\end{enumerate} 
	\end{MainThm}

	The organization of our paper is as follows. The second section is for preliminaries. We also review fold maps and round fold maps. We also define several notions which we have presented in the present section and which we need in more rigorous manners. The third section is devoted to our proofs of Main Theorems.
	The fourth section is devoted to a kind of appendices, explaining about our Main Theorems in the viewpoint of explicit fold maps and algebraic topology and differential topology of manifolds admitting these maps. We see Main Theorems give phenomena in $3$-dimensional manifolds which are already discovered in higher dimensional closed (and simply-connected) manifolds and fold maps on them. In several explicit cases, fold maps of several classes have been shown to distinguish algebraic topologically or differential topologically similar manifolds.
	
	\ \\
	{\bf Conflict of Interest.} \\
	The author is a member of the project JSPS KAKENHI Grant Number JP22K18267 "Visualizing twists in data through monodromy" (Principal Investigator: Osamu Saeki). The present study is due to this project. \\
	\ \\
	{\bf Data availability.} \\
	Data supporting our present study essentially are all in our paper.
	
	\section{Fundamental properties and existing studies on special generic maps and the manifolds.}
	\subsection{Manifolds, differentiable maps and smooth bundles}
	The $k$-dimensional Euclidean space is denoted by ${\mathbb{R}}^k$ for any positive integer $k$. $\mathbb{R}:={\mathbb{R}}^1$ is for the line. ${\mathbb{R}}^2$ is for the plane. Of course $\mathbb{Z} \subset \mathbb{Q} \subset \mathbb{R}$, which is very fundamental. The Euclidean space is also a Riemannian manifold whose underlying metric is the standard Euclidean metric. Let $||x|| \geq 0$ denote the distance between $x \in {\mathbb{R}}^k$ and the origin $0 \in {\mathbb{R}}^k$. Let the $k$-dimensional unit sphere be denoted by $S^k:=\{x \in {\mathbb{R}}^{k+1} \mid ||x||=1\}$ for any integer $k \geq 0$. Let the $k$-dimensional unit disk be denoted by $D^k:=\{x \in {\mathbb{R}}^{k} \mid ||x|| \leq 1\}$ for any integer $k \geq 1$. Their topologies can be understood easily and they are $k$-dimensional smooth compact submanifolds in the Euclidean spaces.   
	
	A topological manifold is regarded as a CW complex. A smooth manifold is regarded as a polyhedron, which is defined uniquely as an object of the PL category, or equivalently, one of the piecewise smooth category. This is a so-called PL manifold. For polyhedra, and, more precisely, topological spaces regarded as CW complexes, for example, we can define their dimensions uniquely. For such a space $X$, let $\dim X$ denote its dimension. 
	
 For differentiable map $c:X \rightarrow Y$ between differentiable manifolds $X$ and $Y$, a {\it singular} point $x \in X$ is defined as a point where the rank of the differential ${dc}_{x}$ is smaller than the minimum between $\dim X$ and $\dim Y$. Let $S(c)$ denote the set of all singular points of $c$ and we call this the {\it singular set} of $c$.
  \begin{Def}
 	\label{def:1}
 	A smooth map $c:X \rightarrow Y$ from a smooth manifold $X$ with no boundary into another smooth manifold $Y$ with no boundary with $\dim X \geq \dim Y$ is said to be a {\it fold} map if at each singular point $p$, it is represented by the form $(x_1,\cdots,x_{\dim X}) \mapsto (x_1,\cdots,x_{\dim Y-1},{\Sigma}_{j=1}^{\dim X-\dim Y-i(p)+1} {x_{\dim Y+j-1}}^2-{\Sigma}_{j=1}^{i(p)} {x_{\dim X-i(p)+j}}^2)$ for suitable local coordinates and some integer $0 \leq i(p) \leq \frac{\dim X-\dim Y+1}{2}$. 
 \end{Def}
Morse functions are fold maps. We also need arguments on Morse functions in our paper. For this, see \cite{milnor2, milnor3} for example. \cite{golubitskyguillemin} explains about elementary and some advanced studies on singularity theory of differentible maps and fundamental expositions on singularities of fold maps. \cite{saeki1} is a pioneering paper on some explicit relations between fold maps and closed manifolds admitting them.
\begin{Prop}
	\label{prop:1}
	For a fold map in Definition \ref{def:1}, $i(p)$ is uniquely defined as the {\rm index} of $p$ and we can define $F_i(c)$ as the set of all singular points whose indices are $i$. 
	The singular set $S(c)$ and the set $F_i(c)$ are {\rm (}$\dim Y-1${\rm )}-dimensional smooth regular submanifolds with no boundaries and the restrictions there give smooth immersions. If $X$ is closed, then these submanifolds are compact.
\end{Prop}
A {\it diffeomorphism} means a smooth map between smooth manifolds which are homeomorphisms and which have no singular points. A {\it diffeomorphism} on a smooth manifold means a diffeomorphism from the manifold onto itself. A {\it diffeomorphism group} on a smooth manifold is a topological group consisting of all diffeomorphisms on the manifold endowed with the {\it Whitney $C^{\infty}$ topology}. Such topologies are natural topologies on spaces of smooth maps between smooth manifolds. See \cite{golubitskyguillemin} again for example.

A {\it smooth} bundle means a bundle whose fiber is a smooth manifold and whose structure group is (a subgroup of) the diffeomorphism group of the fiber. {\it Trivial} bundles are of course important where we do not restrict classes of bundles to the class of smooth bundles. An important subclass is the class of {\it linear} bundles. They are smooth bundles whose fibers are Euclidean spaces, unit spheres, or unit disks, and whose structure groups consist of linear transformations, defined naturally.
For bundles, see \cite{milnorstasheff,steenrod} for example.
\subsection{Round fold maps.}
\begin{Def}
	Let $m \geq n \geq 2$ be integers.
	A {\it round} fold map $f:M \rightarrow {\mathbb{R}}^n$ on a closed and connected manifold $M$ is a fold map enjoying the following properties.
	\begin{enumerate}
		\item The restriction $f {\mid}_{S(f)}$ is an embedding. 
		\item For some diffeomorphism ${\phi}_{{\mathbb{R}}^n}:{\mathbb{R}}^n \rightarrow {\mathbb{R}}^n$ as some integer $l>0$, $({\phi}_{{\mathbb{R}}^n} \circ f)(S(f))=\{x \in {\mathbb{R}}^n \mid 1 \leq ||x|| \leq l, ||x|| \in \mathbb{Z}\}$.
	\end{enumerate}  
\end{Def}
Hereafter, we consider a round fold map $f$ satisfying $f(S(f))=\{x \in {\mathbb{R}}^n \mid 1 \leq ||x|| \leq l, ||x|| \in \mathbb{Z}\}$. This has no problem.
Note that for $n=1$, we can also define a {\it round} fold map as in \cite{kitazawa0.5}. It is defined as a function obtained by gluing two copies of a Morse function satisfying some natural conditions on the boundaries. It is a so-called {\it twisted double}. However we do not need such functions here.

We can define two important classes of round fold maps.
Hereafter, ${D^n}_a:=\{x \in {\mathbb{R}}^n \mid ||x|| \leq a\}$ for $a>0$ and it is diffeomorphic to the $n$-dimensional unit disk $D^n$.
For $0<a_1<a_2$ with $n \geq 2$, ${A^n}_{a_1,a_2}:={D^n}_{a_2}-{\rm Int}\ {D^n}_{a_1}$, which is diffeomorphic to $S^{n-1} \times D^1$.
\begin{Def}
 Let $f:M \rightarrow {\mathbb{R}}^n$ be a round fold map on an $m$-dimensional closed and connected manifold with $m \geq n \geq 2$. Suppose that the number of connected components of the singular set $S(f)$ is $l>0$. 
 \begin{enumerate}
 	\item We can consider the restriction of $f$ to the preimage $f^{-1}({A^n}_{\frac{1}{2},l+\frac{1}{2}})$. We can also compose this with the canonical projection to $S^{n-1}$ mapping each point $x \in {A^n}_{\frac{1}{2},l+\frac{1}{2}}$ to $\frac{1}{||x||}x \in S^{n-1}$.
 	This gives a smooth bundle. If this gives a trivial one, then $f$ is said to {\it have a globally trivial monodromy}.
 \item For each connected component of the set $f(S(f))$, represented by $\partial {D^n}_{l^{\prime}}$ for some integer $1 \leq l^{\prime} \leq l$, we can have a closed tubular neighborhood ${A^n}_{l^{\prime}-\frac{1}{2},l^{\prime}+\frac{1}{2}}$ and consider the restriction of $f$ to the preimage of this closed tubular neighborhood. We can also compose this with the canonical projection to $S^{n-1}$ mapping each point $x \in {A^n}_{l^{\prime}-\frac{1}{2},l^{\prime}+\frac{1}{2}}$ to $\frac{1}{||x||}x \in S^{n-1}$. If this gives a trivial one for each connected component of the set $f(S(f))$, then $f$ is said to {\it have componentwisely trivial monodromies}.
 \end{enumerate} 
\end{Def}
Canonical projections of unit spheres into the Euclidean spaces (whose dimensions are at least $2$) are round fold maps having globally trivial monodromies. To check this is regarded as a kind of fundamental exercises on smooth manifolds and maps, theory of Morse functions, and singularity theory of differentiable maps.

A {\it homotopy sphere} means a smooth manifold homeomorphic to a unit sphere whose dimension is at least $1$. It is said to be a {\it standard} sphere if it is diffeomorphic to the unit sphere and it is said to be an {\it exotic} sphere if it is not.
\cite{kervairemilnor,milnor1} are on such homotopy spheres. It is well-known that $4$-dimensional exotic spheres are undiscovered and that homotopy spheres whose dimensions are not $4$ are completely classified via algebraic topological and abstract theory.
\begin{Prop}[\cite{saeki2} etc.]
\label{prop:2}
	A homotopy sphere whose dimension is at least $2$ and which is not a $4$-dimensional exotic sphere admits a round fold map into ${\mathbb{R}}^2$ whose singular set is connected. If a manifold whose dimension is $m \geq 2$ admits a round fold map into ${\mathbb{R}}^n$ whose singular set is connected satisfying $m \geq n \geq 2$, then it is a homotopy sphere which is not a $4$-dimensional exotic sphere.
\end{Prop}
Note again that round fold maps have been first introduced by the author in \cite{kitazawa0.1,kitazawa0.2,kitazawa0.3} after \cite{saeki2}.
\begin{Thm}[\cite{kitazawa0.1,kitazawa0.2}]
	\label{thm:1}
	Let $m \geq n \geq 2$ be integers. Let $\Sigma$ be a homotopy sphere which is not a $4$-dimensional exotic sphere.
	An $m$-dimensional closed manifold $M$ admits a round fold map $f:M \rightarrow {\mathbb{R}}^n$ enjoying the following properties if and only if it is the total space of a smooth bundle over the $n$-dimensional unit sphere $S^n$ whose fiber is $\Sigma$.
	\begin{enumerate}
		\item $f$ has a globally trivial monodromy.
		\item $S(f)$ consists of exactly two connected components. The index of each singular point is $0$ or $1$. $f(F_1(f))= \partial {D^n}_1$ and $f(F_0(f))= \partial {D^n}_2$.
		\item For a point in ${D^{n}}_1-f(S(f))$, the preimage is the disjoint union of two copies of $\Sigma$.
		\item For a point in ${A^{n}}_{1,2}-f(S(f))$, the preimage is an {\rm (}$m-n${\rm )}-dimensional standard sphere.
	\end{enumerate}
\end{Thm}
Proofs of this are in refereed articles \cite{kitazawa0.1, kitazawa0.4}, the doctoral dissertation \cite{kitazawa0.2} and a preprint \cite{kitazawa0.5} of the author. We present our proof again.
\begin{proof}[A proof of Theorem \ref{thm:1}]
	We show the "if" part.
	The base space
	$S^n$ is decomposed into the following two manifolds and we can reconstruct $S^n$ by gluing them by a diffeomorphism along the boundaries.
	\begin{itemize}
		\item $D^{n,1} \sqcup D^{n,2}$, which denotes the disjoint union of two smoothly embedded copies of the $n$-dimensional unit disk. 
		\item $\partial D^{n,0} \times [-1,1]$, where $\partial D^{n,0}$ denotes the boundary of $D^{n,1}$ or $D^{n,2}$.
		\end{itemize}
	We can have $D^{n,1} \sqcup D^{n,2}$ as the preimage of a double cover over ${D^n}_{\frac{1}{2}}$. We can construct the projection over $D^{n,1} \sqcup D^{n,2}$ whose preimage is diffeomorphic to $\Sigma$. In other words, we also have a trivial smooth bundle over $D^{n,1} \sqcup D^{n,2}$ whose fiber is $\Sigma$ or one over ${D^n}_{\frac{1}{2}}$ whose fiber is the disjoint union $\Sigma \sqcup \Sigma$.
	
	We can construct the projection over
	$\partial D^{n,0} \times [-1,1]$ whose preimage is diffeomorphic to $\Sigma$. Over the total space of the resulting trivial bundle, we can have the product map of a natural Morse function with exactly two singular points over the cylinder $\Sigma \times [-1,1]$  and the identity map on $\partial D^{n,0}$. The Morse function enjoys properties that the preimage of the minimum (maximum) and the boundary coincide and that the two singular points are all in the interior for example.
	
	We can regard this product map as a map into $A_{\frac{1}{2},\frac{5}{2}}$ whose image is $A_{\frac{1}{2},2}$. 
	
	We can glue the two maps to obtain a desired round fold map. This completes the proof of the "if" part.
	
	The "only if" part can be presented shortly. 
	We abuse the notation before in a natural way. Conversely, we can decompose the $m$-dimensional manifold into two manifolds which are the total spaces of smooth bundles over $D^{n,1} \sqcup D^{n,2}$ and $\partial D^{n,0} \times [-1,1]$. Moreover fibers are diffeomorphic to $\Sigma$ of course. 
	By observing the identification to have the base space, which is a homotopy sphere, carefully, we can see that the base space must be a standard sphere. This completes the proof of the "only if" part.
	
	This completes the proof.
\end{proof}

\begin{Thm}[\cite{kitazawa0.3,kitazawa0.5}]
\label{thm:2}
	Let $m \geq n \geq 2$ and $l>0$ be integers. Let an $m$-dimensional closed manifold $M$ be represented as a connected sum of $l>0$ manifolds each of which is the total space of some smooth bundle over $S^n$ whose fiber is an {\rm (}$m-n${\rm )}-dimensional standard sphere. 
	Here a connected sum is taken in the smooth category.
	$M$ admits a round fold map $f:M \rightarrow {\mathbb{R}}^n$ enjoying the following properties if and only if it is the total space of a smooth bundle over the $n$-dimensional unit sphere $S^n$ whose fiber is $\Sigma$.
\begin{enumerate}
	\item $f$ has componentwisely trivial monodromies.
	\item $S(f)$ consists of exactly $l+1$ connected components. The index of each singular point is $0$ or $1$. $f(F_1(f))={\subset}_{j=1}^{l-1} \partial {D^n}_j$ and $f(F_0(f))= \partial {D^n}_l$.
	\item For a point in ${D^{n}}_1-f(S(f))$, the preimage is the disjoint union of $l+1$ {\rm (}$m-n${\rm )}-dimensional standard spheres.
	\item For a point in ${A^{n}}_{l^{\prime},l^{\prime}+1}-f(S(f))$, the preimage is the disjoint union of $l+1-l^{\prime}$ {\rm (}$m-n${\rm )}-dimensional standard spheres for each $1 \leq l^{\prime} \leq l$.
\end{enumerate}
In addition, if $m \geq 2n$ is assumed, then the converse also holds. If it is not assumed, then the converse does not hold in general and the case $(m,n)=(3,2)$ gives one of such cases. 
\end{Thm}
\begin{Def}
	We say round fold maps as in Theorem \ref{thm:2} are {\it directed}. We say ones whose singular sets are connected and which have globally trivial monodromies are also {\it directed}.
\end{Def}

We explain about directed round fold maps whose singular sets are connected.
For such maps, preimages of points in the interiors of the images are standard spheres. Furthermore, the index of singular points of the maps here are always $0$. Canonical projections of unit spheres into Euclidean spaces whose dimensions are at least $2$ satisfy the conditions.

\begin{Prop}
	\label{prop:3}
	In Theorem \ref{thm:2} consider the case $l=1$. Directed round fold maps have globally trivial monodromies in the case $(m,n)=(k+1,k),(4,2), (5,2), (5,3), (6,3), (6,4)$ for example where $k$ is an arbitrary integer greater than or equal to $2$.
\end{Prop}
\begin{proof}[A short exposition on Proposition \ref{prop:3}.]
	We abuse the notation $l+1:=2>0$ for the number of connected components of the singular set $S(f)$ for a directed round fold map $f$. 
	We consider the restriction of $f$ to $f^{-1}({A^n}_{l+\frac{1}{2},l+\frac{3}{2}})$. We can also compose this with the canonical projection to $S^{n-1}$ mapping each point $x \in {A^n}_{l+\frac{1}{2},l+\frac{3}{2}}$ to $\frac{1}{||x||}x \in S^{n-1}$. This bundle is a linear bundle whose fiber is the unit disk $D^{m-n+1}$ according to \cite{saeki2}. We can also say that in our case, this is also trivial as a linear bundle. Smooth bundles whose fibers are diffeomorphic to $S^{m-n}$ are in our case linear where in general this is not true of course. In our case, we can see that an isomorphism on the smooth (linear) bundle $f^{-1}(\partial {D^n}_{l+\frac{1}{2}})$ is always extended to an isomorphism on the smooth (linear) bundle over $f^{-1}({A^n}_{l+\frac{1}{2},l+\frac{3}{2}})$. This is a most important ingredient and this completes our exposition on Proposition \ref{prop:3}.
	
	For linear bundles related to our arguments here, consult \cite{milnorstasheff} again. \cite{hatcher1} is on the diffeomorphism groups of unit spheres and smooth bundles whose fibers are standard spheres.  
\end{proof}
\begin{Prop}
	\label{prop:4}
In the definition of directed round fold maps, we do not need the assumption that maps have componentwisely trivial monodromies in the case $(m,n)=(k+1,k)$ for any integer $k \geq 2$.
\end{Prop}
This is due to a well-known fact on the diffeomorphisms groups, or the so-called {\it  mapping class groups} of compact surfaces and some related arguments are important imgredients of \cite{kitazawasaeki2}. Proposition \ref{prop:6} is presented later as a well-known fact on the mapping class group and the diffeomorphism group of the torus $S^1 \times S^1$. This is also an important fact in our arguments.
	\section{On Main Theorems.}
	\subsection{Graph manifolds and round fold maps into ${\mathbb{R}}^2$ on them.}
	We explain about graph manifolds. Note that topological manifolds whose dimensions are at most $3$ are all smooth manifolds and for a fixed topological manifold here, the smooth manifolds are always mutually diffeomorphic. This is due to \cite{moise} for example.
	\begin{Def}
		\label{def:5}
	A {\it graph manifold} is a $3$-dimensional closed, connected and orientable manifold obtained from finitely many manifolds regarded as the total spaces of smooth bundles over compact and connected surfaces whose fibers are diffeomorphic to $S^1$ or {\it circle bundles} over the surfaces by gluing tori on the boundaries by diffeomorphisms.
	\end{Def}
	\begin{Thm}
		\label{thm:3}
	A $3$-dimensional closed, connected and orientable manifold admits a round fold map into ${\mathbb{R}}^2$ if and only if it is a graph manifold.
	\end{Thm}
    Originally, in \cite{saeki3}, "a round fold map" is "a fold map into ${\mathbb{R}}^2$ such that the restriction to the singular set is an embedding".
    
    An $m$-dimensional {\it pair of pants} is a smooth manifold diffeomorphic to one obtained by removing three disjointly and smoothly embedded copies of the $m$-dimensional unit sphere from an $m$-dimensional standard sphere.
    \begin{Prop}
    	\label{prop:5}
    	In Definition \ref{def:5}, we can choose each bundle as a trivial bundle over the unit disk $D^2$ or a $2$-dimensional pair of pants.
    	We can also define a graph satisfying the following rules.
    	\begin{enumerate}
    		\item The vertex set is the set of all these circle bundles over the surfaces.
    		\item The edge set is the set of all tori of connected components of the boundaries of these circle bundles over the surfaces.
    		\item Each edge contains exactly two vertices and they are distinct. They are the circle bundles containing the edge or the torus as a connected component of the boundaries.    	\end{enumerate} 
    \end{Prop}
Related to Proposition \ref{prop:5}, we explain about graphs associated with graph manifolds shortly where we do not need deep understanding on these graphs in our paper. 
We call a graph in Proposition \ref{prop:5} a {\it representation graph} for the graph manifold $M$. It is not unique. In our studies, representation graphs of a certain type or the so-called {\it plumbing type} are important. This is for decompositions into circle bundles where connected components of boundaries are glued in some specific ways. See \cite{neumann, saeki1} and see also \cite{kitazawasaeki1}. 
For each graph manifold, graphs with several labels are defined in \cite{neumann}. The definition of such a graph with labels is more complicated. As a strong result, we can define a {\it normal form} in such graphs. This also gives an invariant for graph manifolds.

The following theorem shows that both these graphs are important in characterizing a same subclass of graph manifolds.

    \begin{Thm}[\cite{kitazawasaeki1}]
    	\label{thm:4}
    	\begin{enumerate}
    		\item A graph manifold admits a directed round fold map into ${\mathbb{R}}^2$ if and only if its normal form is a graph with no cycles. 
    	\item Equivalently, a graph manifold admits a directed round fold map into ${\mathbb{R}}^2$ if and only if there exists a representation graph with no cycles for the manifold. 
    	\end{enumerate}
    \end{Thm}
\begin{Thm}[\cite{doighorn,kitazawasaeki1}]
	\label{thm:5}
	For manifolds in the previous theorem, the cup product for any ordered pair of elements of the $1$st rational cohomology classes is always the zero element.  
\end{Thm}
In \cite{kitazawasaeki1}, the fact that graph manifolds admitting directed fold maps into ${\mathbb{R}}^2$ enjoy such a property on the rational cohomology rings is presented as a corollary to a main result of \cite{doighorn}. The result of \cite{doighorn} states that graph manifolds with normal forms having no cycles enjoy this property.
Main Theorem \ref{mthm:1} is regarded as a stronger version. 
\subsection{Proofs of Main Theorems.}
We need some notions on polyhedra and so-called {\it Reeb spaces}.

A topological space which is homeomorphic to a polyhedron whose dimension is at most $2$ has the structure of a polyhedron uniquely. A topological space which is homeomorphic to a topological manifold whose dimension is at most $3$ has the structure of a polyhedron uniquely and this is the PL manifold. This is also due to \cite{moise} for example.

	For a continuous map $c:X \rightarrow Y$ between topological spaces, we can define an equivalence relation ${\sim}_c$ on $X$ by the following rule: $x_1 {\sim}_c x_2$ if and only if $x_1$ and $x_2$ are in a same connected component of the preimage $c^{-1}(y)$ of some point $y$.
	
We do not explain about general theory of Reeb spaces precisely. One of important fact is that for fold maps and more general smooth maps enjoying some properties on so-called "genericity", they are polyhedra whose dimensions are same as those of the manifolds of the targets and whose structures as the polyhedra are naturally induced from the manifolds of the targets. Such facts are shown in \cite{shiota} for example. \cite{kobayashisaeki} explicitly shows such a fact for so-called {\it stable} maps on smooth closed manifolds whose dimensions are at least $3$ into surfaces with no boundaries. Essentially the class of such maps there contains round fold maps and fold maps such that the restrictions to the singular sets are embeddings for example. See \cite{golubitskyguillemin} for related singularity theory of smooth maps. 

For such polyhedra, see also \cite{turaev} for example. This is a paper published before studies of such polyhedra regarded as the Reeb spaces of such smooth maps into ${\mathbb{R}}^2$ started. For related studies, see \cite{costantinothurston,ishikawakoda}.

The following example is important.

It is well-known that homotopy spheres except $4$-dimensional exotic spheres are all PL homeomorphic to standard spheres where they are seen as the PL manifolds. We call such PL manifolds {\it PL spheres}.

\begin{Ex}
	For example, for directed round fold maps on $m$-dimensional closed and connected manifolds into ${\mathbb{R}}^n$ with $m>n$, the Reeb spaces are simple homotopy equivalent to the bouquet of $n$-dimensional (PL) spheres. 
	FIGURE \ref{fig:1} is for Theorem \ref{thm:1}. This is obtained by attaching a copy ${{D^n}_0}^{\prime}$ of the $n$-dimensional unit sphere ${D^n}_{0}$ to another copy via a diffeomorphism from the boundary onto a smoothly embedded copy of the unit sphere $S^{n-1}$ in the interior of ${D^n}_{0}$. FIGURE \ref{fig:2} is for a general case. Note that topologically we can represent the Reeb space as a one embedded naturally into ${\mathbb{R}}^{n+1}$ and this is the subspace of the Reeb space in some hyperplane of ${\mathbb{R}}^{n+1}$.  
\end{Ex}
\begin{figure}
	
\includegraphics[height=25mm, width=40mm]{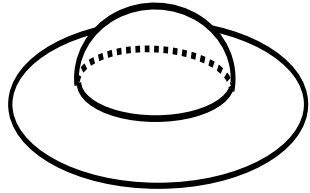}
\begin{center}
\caption{The Reeb space for Theorem \ref{thm:1}.}
\label{fig:1}
\end{center}
\end{figure}
\begin{figure}

	\includegraphics[height=25mm, width=40mm]{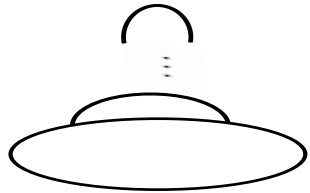}
	\begin{center}
	\caption{The Reeb space of a general directed fold map. The subspace of it in a suitable hyperplane of ${\mathbb{R}}^{n+1}$}
		\label{fig:2}
	\end{center}
\end{figure}
Let $A$ be a commutative ring with a generator $\pm a \in A$.
For a compact, connected and oriented manifold $X$, we can define the fundamental class as a generator of $H_{\dim X}(X,\partial X;A)$, isomorphic to $A$. We can set this element as $\pm a$ according to the orientation.

For a smooth manifold $Y$ and an element $c \in H_j(Y,\partial Y,A)$, assume that $c$ is equal to the value ${i_X}_{\ast}(a)$ of the homomorphism ${i_{X}}_{\ast}:H_{\dim X}(X;\partial X;A) \rightarrow H_{j}(Y;\partial Y;A)=H_{\dim X}(Y;\partial Y;A)$ induced canonically by some smooth embedding $i_X:X \rightarrow Y$ of a smooth, compact, connected and oriented manifold $X$ enjoying the following properties.
\begin{itemize}
	\item $i_X(\partial X) \subset \partial Y$.
	\item $i_X({\rm Int}\ X) \subset {\rm Int}\ Y$.
\end{itemize}
In other words, they are embedded properly. 
Respecting fundamental arguments on differential topology, we may add so-called "transversality" of the embedding on the boundaries according to the situations.
In other words, at each point $p$ in the boundary $\partial X$, the intersection of the image of the differential ${d_{i_X}}_p$ of the embedding and the tangent space $T_{i_X(p)}Y$ of $Y$ at $i_X(p)$ is assumed to be of dimension $\dim X+(\dim Y-1)-\dim Y=\dim X-1$. 
However, we do not consider this assumption essentially in our paper.
In this situation, $c$ is said to be {\it represented by} the submanifold {\it $i_X(X)$}.  

We can define similar notions in the PL category, or equivalently, the piecewise smooth category, and the topology category.


Consider a compact, connected and oriented manifold $X$ again. In our arguments, so-called Poincar\'e duals to elements of $H_j(X;\partial X;A)$, $H^j(X;\partial X;A)$. $H_j(X;A)$ and $H^j(X;A)$ are important. They are uniquely defined as elements of $H^{\dim X-j}(X;A)$, $H_{\dim X-j}(X;A)$. $H^{\dim X-j}(X,\partial X;A)$ and $H_{\dim X-j}(X,\partial X;A)$, respectively. Related to this,
Poincar\'e duality theorem for $X$ is also important. 

Of course we calculate (co)homology groups and cohomology rings. For this, we need exact sequences such as Mayer-Vietoris sequences. We also need some theorems such as K\"unneth formula, useful for the cohomology groups (rings) of products. Poincar\'e duality theorem, presented before, and universal coefficient theorem are important. See \cite{hatcher1} again for example. 

We show Main Theorems. 
First, the following proposition is fundamental in our arguments. 

\begin{Prop}
\label{prop:6}
For a trivial circle bundle $T^2:=S^1 \times S^1$ over $S^1$, there exists a family of isomorphisms on the bundle, which can be denoted by ${\{{\Phi}_j\}}_{j \in \mathbb{Z}}$ and enjoy the following properties.
	\begin{enumerate}
		\item Let ${S^1}_{\rm b}$ denote the subspace $S^1 \times \{\ast\} \subset S^1 \times S^1$ with an orientation. Let ${S^1}_{\rm f}$ denote the subspace $\{\ast\} \times S^1 \subset S^1 \times S^1$ with an orientation, which is regarded as a fiber of the trivial circle bundle over ${S^1}_{\rm b}$. Let $[S_{\rm b}] \in H_1(T^2;\mathbb{Z}) \oplus \mathbb{Z} \oplus \mathbb{Z}$ and $[S_{\rm f}] \in H_1(T^2;\mathbb{Z}) \oplus \mathbb{Z} \oplus \mathbb{Z}$ denote the elements represented by these oriented circles.
		The homomorphism ${{\Phi}_j}_{\ast}$ induced from ${\Phi}_j$, which is also an isomorphism, maps $[S_{\rm b}]$ to $[S_{\rm b}]+j[S_{\rm f}]$ and $[S_{\rm f}]$ to $[S_{\rm f}]$.
		\item 
		Any isomorphism on the bundle is smoothly isotopic to exactly one ${\Phi}_j$ and we can smoothly isotope in the space of isomorphisms on the bundle.
	\end{enumerate}
\end{Prop}
We omit rigorous exposition on this. 
We also need closely related arguments in our paper. For these arguments, we use fundamental knowledge and methods on linear bundles including circle bundles. When we need, see \cite{milnorstasheff} again for example. Although terminologies and notation may be a bit different from ours, we can consult this, explaining about the essentially same content.

For Main Theorem \ref{mthm:1}, the following lemma is essential.
\begin{Lem}
\label{lem:1}
Assume that a $3$-dimensional closed, connected and orientable manifold $M$ admits a directed round fold map $f:M \rightarrow {\mathbb{R}}^2$. Then we have the following properties.
\begin{enumerate}
\item
\label{lem:1.1}
 The 2nd integral homology group $H_2(M;\mathbb{Z})$ is free and generated by finitely many elements which are not divisible by integers greater than $1$. 

\item \label{lem:1.2} Furthermore, the set of the finitely many elements in $H_2(M;\mathbb{Z})$ can be taken as its basis. In addition, these elements can be taken as ones represented by spheres which are mutually disjoint.
Let $\{{S^2}_j\}$ denote the set of all spheres here.
\item \label{lem:1.3} We can consider the quotient map $q_f:M \rightarrow W_f$ onto the Reeb space $W_f$ of $f$. Then $q_f$ maps ${S^2}_j$ onto a PL sphere by a {\rm (}PL{\rm )} homeomorphism. Furthermore, the
 homomorphism ${q_f}_{\ast}:H_2(M;\mathbb{Z}) \rightarrow H_2(W_f;\mathbb{Z})$ induced canonically by the map $q_f$ is a monomorphism and maps each element represented by ${S^2}_j$ to an element which is not divisible by integers greater than $1$. 
\item \label{lem:1.4} The subgroup of 1st integral homology group $H_1(M;\mathbb{Z})$ generated by all elements whose orders are infinite is generated by finitely many elements which are not divisible by integers greater than $1$. 
\item \label{lem:1.5} Furthermore, the set of the finitely many elements in $H_1(M;\mathbb{Z})$ can be taken as a basis of the subgroup of the 1st integral homology group $H_1(M;\mathbb{Z})$ generated by all elements whose orders are infinite. In addition, these elements are represented by connected components of the preimage $f^{-1}(0)$.
\end{enumerate} 

	\end{Lem}

\begin{proof}
	We prove by an induction on the numbers of connected components of the singular sets.
If the number is $1$, then by Proposition \ref{prop:2}, the manifold is a $3$-dimensional (standard) sphere. Our lemma holds of course.

Suppose that our lemma holds if the numbers of connected compoents of the singular sets are at most $k$ where $k$ is a positive integer.

We consider a directed round fold map $f_{{\rm r},k+1}:{M^3}_{k+1} \rightarrow {\mathbb{R}}^2$ such that the singular set $S(f_{{\rm r},k+1})$ consists of exactly $k+1$ connected components.
Remove a connected component $B_{k}$ of the inteiror of the preimage ${f_{{\rm r},k+1}}^{-1}({D^2}_{\frac{3}{4}})$, regarded as the total space of a trivial circle bundle over ${D^2}_{\frac{3}{4}}$. Let $A_{k}$ denote the resulting $3$-dimensional compact, connected and orientable manifold. The resulting map is regarded as the restriction of a round fold map on a $3$-dimensional closed, connected and orientable manifold $A_{k,0}$ into ${\mathbb{R}}^2$ obtained by removing the interior of a copy of $S^1 \times D^2$ smoothly embedded in $A_{k,0}$. The round fold map is, by Proposition \ref{prop:3} and arguments in Proposition \ref{prop:4}, deformed to a directed round fold map $f_{{\rm r},k,0}:A_{k,0} \rightarrow {\mathbb{R}}^2$ by a natural smooth homotopy eliminating two adjacent two connected components of the singular set, consisting of singular points whose indices are $0$ and $1$, respectively.
We can consider a Mayer-Vietoris sequence 

$$ \rightarrow H_3(A_k;\mathbb{Q}) \oplus H_3(S^1 \times D^2;\mathbb{Q}) \cong \{0\} \rightarrow H_3(A_{k,0};\mathbb{Q}) \cong \mathbb{Q} \rightarrow$$ 
$$\rightarrow H_2(S^1 \times S^1;\mathbb{Q}) \cong \mathbb{Q} \rightarrow H_2(A_k;\mathbb{Q}) \oplus H_2(S^1 \times D^2;\mathbb{Q}) \cong H_2(A_k;\mathbb{Q}) \rightarrow H_2(A_{k,0};\mathbb{Q}) \rightarrow$$ 
 
and we can consider a family $\{{S^2}_{k,j}\}$ of spheres in $A_{k,0}$ as in our lemma for the directed round fold map. We can see that the homomorphism from $H_2(A_k;\mathbb{Q})$ into $H_2(A_{k,0};\mathbb{Q})$ here is a monomorphism and induced by
 the inclusion canonically. By the construction of the round fold maps and the manifolds, the spheres in $\{{S^2}_{k,j}\}$ can be also regarded as ones in ${\rm Int}\ A_k$. Furthermore, they can be regarded as ones mapped by the inclusion and by diffeomorphisms onto the corresonding spheres in $A_{k,0}$. This also means that the ranks of $H_2(A_k;\mathbb{Q})$, $H_2(A_{k,0};\mathbb{Q})$, $H_2(A_k;\mathbb{Z})$ and $H_2(A_{k,0};\mathbb{Z})$ agree
 by universal coefficient theorem. 

We can consider a Mayer-Vietoris sequence 

$$ \{0\}\rightarrow H_3({M^3}_{k+1};\mathbb{Q}) \cong \mathbb{Q} \rightarrow$$ 
$$\rightarrow H_2(S^1 \times S^1;\mathbb{Q}) \cong \mathbb{Q} \rightarrow H_2(A_k;\mathbb{Q}) \oplus H_2(B_k;\mathbb{Q}) \cong H_2(A_k;\mathbb{Q}) \rightarrow H_2({M^3}_{k+1};\mathbb{Q}) \rightarrow$$ 
$$\rightarrow H_1(S^1 \times S^1;\mathbb{Q}) \cong \mathbb{Q} \oplus \mathbb{Q}  \rightarrow H_1(A_k;\mathbb{Q}) \oplus H_1(B_k;\mathbb{Q}) \cong H_1(A_k;\mathbb{Q}) \oplus \mathbb{Q} \rightarrow H_1({M^3}_{k+1};\mathbb{Q}) \rightarrow$$ 

and we first explain about the homomorphism from 
$H_1(S^1 \times S^1;\mathbb{Q}) \cong \mathbb{Q}$ into $H_1(A_k;\mathbb{Q}) \oplus H_1(B_k;\mathbb{Q}) \cong H_1(A_k;\mathbb{Q}) \oplus \mathbb{Q}$. This is the direct sum of the two homomorphisms induced canonically by the inclusions. 
By considering fibers of the trivial circle bundles, the dimension of the kernel must be $0$ or $1$. For this, remember that $B_k$ is the total space of a trivial circle bundle over ${D^2}_{\frac{3}{4}}$ and an element represented by the fiber of the bundle $S^1 \times S^1$ over $S^1$ is mapped to an element of the form $(c_1,c_{\rm F}) \in H_1(A_k;\mathbb{Q}) \oplus H_1(B_k;\mathbb{Q})$ where $c_{\rm F}$ is represented by a fiber of the circle bundle $B_k$. 

The rank of $H_2({M^3}_{k+1};\mathbb{Q})$ is equal to that of $H_2(A_k;\mathbb{Q})$ or the sum of the rank of $H_2(A_k;\mathbb{Q})$ and $1$. We argue the two cases to complete the proof. \\
\ \\
Case 1 The rank of $H_2({M^3}_{k+1};\mathbb{Q})$ is equal to that of $H_2(A_k;\mathbb{Q})$. \\

The rank of $H_2({M^3}_{k+1};\mathbb{Z})$ is equal to those of $H_2(A_k;\mathbb{Z})$, $H_2({M^3}_{k+1};\mathbb{Q})$ and $H_2(A_k;\mathbb{Q})$
by the universal coefficient theorem. Furthermore, $H^1({M^3}_{k+1};\mathbb{Z})$ and $H_2({M^3}_{k+1};\mathbb{Z})$ are free and isomorphic by universal coefficient theorem and Poincar\'e duality theorem.
Reviewing the construction of the maps and the manifolds, we can see that the round fold map $f_{{\rm r},k+1}:{M^3}_{k+1} \rightarrow {\mathbb{R}}^2$ enjoys the desired properties. \\
\ \\
Case 2 The rank of $H_2({M^3}_{k+1};\mathbb{Q})$ is equal to the sum of the rank of $H_2(A_k;\mathbb{Q})$ and $1$. \\

We assume that this occurs.

By the construction of the maps and the manifolds, there exists a closed, connected and oriented surface $S$ and an element $c_S$ of $H_2({M^3}_{k+1};\mathbb{Q})$ represented by the surface $S$ and elements represented by spheres in the family $\{{S^2}_{k,j}\}$ form a basis of $H_2({M^3}_{k+1};\mathbb{Q})$. 

By an argument on algebraic topology and differential topology, we investigate properties of $S$. $S$ is divided by $\partial A_k$, identified with $\partial B_k$ in a canonical way. More precisely, it is decomposed into two compact surfaces along circles in the boundary. By Poincar\'e duality theorem for $B_k$ or the so-called intersection theory, there exists an element $c_{{\rm F},k+1} \in H_1({M^3}_{k+1};\mathbb{Q})$ enjoying the following properties with the arguments.

\begin{itemize}
\item Let $\{c_{{\rm F},k,j}\} \subset H_1(A_{k,0};\mathbb{Q})$ denote the basis obtained from a basis as in (\ref{lem:1.5}) in the assumption by universal coefficient theorem. These elements are seen as mutually independent in $H_1(A_k;\mathbb{Q})$ by regarding them as elements of $H_1(A_k;\mathbb{Q})$ naturally by respecting the structures of the maps and the manifolds as before. They are also seen as mutually independent in $H_1({M^3}_{k+1};\mathbb{Q})$ by considering the inclusion. Here we also respect intersections for the circles in $f^{-1}(0)$ by which the elements are represented and $2$-dimensional spheres in $\{{S^2}_j\}$ for the 2nd integral (rational) homology group. 
Furthermore, $c_{{\rm F},k+1} \in H_1({M^3}_{k+1};\mathbb{Q})$ and these elements form a basis.
\item $c_{{\rm F},k+1}$ is represented by a fiber of the trivial circle bundle $B_k \subset {M^3}_{k+1}$.     
\end{itemize}

We investigate the unique connected component of the preimage ${f_{{\rm r},k+1}}^{-1}({D^2}_{\frac{3}{2}})$ containing $B_k$ as the subspace. By observing Theorem \ref{thm:1}, Proposition \ref{prop:3} and Proposition \ref{prop:4} and their proofs, we can see that the restriction of the round fold map here is seen as the restriction of a round fold map like one in Theorem \ref{thm:1} to the preimage of ${D^{2}}_{\frac{3}{2}}$.

We explain about circle bundles over $S^2$ shortly. Related to this, circle bundles over the torus are discussed more precisely in the proof of Main Theorem \ref{mthm:2} and this may help us to discuss more precisely.

Each circle bundle over $S^2$ corresponds to, modulo isomorphisms with the structure groups preserving the orientations of fibers, its {\it Euler number} $k \in \mathbb{Z}$.  The total space of such a bundle is a so-called {\it rational homology sphere} or a $3$-dimensional closed manifold whose rational homology group is isomorphic to that of $S^3$ if and only if $k \neq 0$. Consider a round fold map $f_{{S^3}_k}$ as in Theorem \ref{thm:1} on such a rational homology sphere ${S^3}_k$. Let $A_{{S^3}_k}$ denote the preimage of ${D^{2}}_{\frac{3}{2}}$ for the map and $B_{{S^3}_k}:={S^3}_k-{\rm Int}\ A_{{S^3}_k}$.

We can consider a Mayer-Vietoris sequence 

$$ \{0\}\rightarrow H_3({S^3}_k;\mathbb{Q}) \cong \mathbb{Q} \rightarrow$$ 
$$\rightarrow H_2(S^1 \times S^1;\mathbb{Q}) \cong \mathbb{Q} \rightarrow H_2(A_{{S^3}_k};\mathbb{Q}) \oplus H_2(B_{{S^3}_k};\mathbb{Q}) \cong H_2(A_{{S^3}_k};\mathbb{Q}) \rightarrow H_2({S^3}_k;\mathbb{Q}) \cong \{0\} \rightarrow$$ 
$$\rightarrow H_1(S^1 \times S^1;\mathbb{Q}) \cong \mathbb{Q} \oplus \mathbb{Q}  \rightarrow H_1(A_{{S^3}_k};\mathbb{Q}) \oplus H_1(B_{{S^3}_k};\mathbb{Q}) \cong H_1(A_{{S^3}_k};\mathbb{Q}) \oplus \mathbb{Q} \rightarrow H_1({S^3}_k;\mathbb{Q}) \cong \{0\} \rightarrow$$ 
and we investigate the isomorphism from $H_1(S^1 \times S^1;\mathbb{Q}) \cong \mathbb{Q} \oplus \mathbb{Q}$ onto $H_1(A_{{S^3}_k};\mathbb{Q}) \oplus H_1(B_{{S^3}_k};\mathbb{Q}) \cong H_1(A_{{S^3}_k};\mathbb{Q}) \oplus \mathbb{Q}$. As presented, this is defined by using the homomorphisms induced by the inclusions. $H_1(A_{{S^3}_k};\mathbb{Q})$ is isomorphic to $\mathbb{Q}$ and the rank of $H_1(A_{{S^3}_k};\mathbb{Z})$ is $1$ by universal coefficient theorem. Furthermore, a generator of this is represented by the preimage of a point in $\partial {D^{2}}_{\frac{5}{4}}$.

This contradicts the property that a nice element $c_{{\rm F},k+1} \in H_1({M^3}_{k+1};\mathbb{Q})$ is taken. Thus the restriction of the round fold map $f_{{\rm r},k+1}$ to the unique connected component of the preimage ${f_{{\rm r},k+1}}^{-1}({D^2}_{\frac{3}{2}})$ containing $B_k$ cannot be a map like this. 

This must be one obtained from a round fold map on $S^2 \times S^1$ as in Theorem \ref{thm:1}.
 We can see that by the construction in the proof of Theorem \ref{thm:1}, $S$ can be taken as the image of a section of the trivial bundle over the base space $S^2$. By the construction, the desired properties are also enjoyed. 

Last note that the 2nd integral homology group of a $3$-dimensional closed, conencted and orientable manifold is free. This is due to Poicar\'e duality theorem and the universal coefficient theorem for the 1st integral cohomology group, forcing the groups to be free. This completes the proof of Case 2.

This completes the proof. 
\end{proof}

\begin{proof}[A proof of Main Theorem \ref{mthm:1}]
	We apply Lemma \ref{lem:1} with Poicar\'e duality theorem or intersection theory. ${S^2}_j$ is, via a suitable smooth isotopy, moved in such a way that the original sphere and the resulting sphere are mutually disjoint. Distinct spheres in $\{{S^2}_j\}$ are mutually disjoint.
This completes the proof.
 \end{proof}

\begin{proof}[A proof of Main Theorem \ref{mthm:2}]
	
	We consider a circle bundle over the torus $T^2:=S^1 \times S^1$ whose {\it Euler number} is $k \in \mathbb{Z}$. We explain about this. Consider a manifold obtained by removing the interior of a smoothly embedded copy of the $2$-dimensional unit disk $D^2$ in $T^2$, denoted by ${T^2}_o$. Let the removed disk denoted by ${D^2}_o$.

We consider trivial smooth bundles over ${T^2}_o$ and ${D^2}_o$ and the total spaces can be denoted by ${T^2}_o \times S^1$ and ${D^2}_o \times S^1$. Their trivializations are naturally given. On their boundaries the trivializations are also given and identified with a trivial bundle $S^1 \times S^1$ over $S^1$ naturally. We glue ${T^2}_o \times S^1$ and ${D^2}_o \times S^1$ by a bundle isomorphism ${\Phi}_k$ in Proposition \ref{prop:5} from $\partial {T^2}_o \times S^1$ onto $\partial {D^2}_o \times S^1$ where the bundles of the domain and that of the target are identified with the bundle $S^1 \times S^1$ canonically as before. We have a desired circle bundle and its total space $B_{S^1,k}(T^2)$.

We have a Mayer-Vietoris sequence
$$\rightarrow H_1(S^1 \times S^1;\mathbb{Z}) \rightarrow H_1({T^2}_o \times S^1;\mathbb{Z}) \oplus H_1({D^2}_o \times S^1;\mathbb{Z}) \rightarrow H_1(B_{S^1,k}(T^2);\mathbb{Z}) \rightarrow$$ 
and we have $H_1({T^2}_o \times S^1;\mathbb{Z}) \cong H_1({T^2} \times S^1;\mathbb{Z}) \cong \mathbb{Z} \oplus \mathbb{Z} \oplus \mathbb{Z}$ and $H_1({D^2}_o \times S^1;\mathbb{Z}) \cong \mathbb{Z}$ for example. 
In our arguments here, we can apply suitable identifications of the homology groups and the finitely generated (commutative) groups and we apply.
The first homomorphism from $H_1(S^1 \times S^1;\mathbb{Z})$ into $H_1({T^2}_o \times S^1;\mathbb{Z}) \oplus H_1({D^2}_o \times S^1;\mathbb{Z})$ is regarded as a homomorphism mapping $(a,b) \in \mathbb{Z} \oplus \mathbb{Z} \cong H_1(S^1 \times S^1;\mathbb{Z})$ to $(0,0,b,ka+b)$. 
Remember that this is the direct sum of the two homomorphisms induced canonically by the inclusions.
The second homomorphism from $H_1({T^2}_o \times S^1;\mathbb{Z}) \oplus H_1({D^2}_o \times S^1;\mathbb{Z})$ into $H_1(B_{S^1,k}(T^2);\mathbb{Z})$ is defined as the sum of the two homomorphisms induced canonically by the inclusions into $B_{S^1,k}(T^2)$. Furthemore, this is an epimorphism. This is due to the fact that $S^1 \times S^1$ is connected and after the third group here, the homomorphism from $H_0(S^1 \times S^1;\mathbb{Z})$ to $H_0({T^2}_o \times S^1;\mathbb{Z}) \oplus H_0({D^2}_o \times S^1;\mathbb{Z})$ follows. Furthermore, this homomorphism from $H_0(S^1 \times S^1;\mathbb{Z})$ to $H_0({T^2}_o \times S^1;\mathbb{Z}) \oplus H_0({D^2}_o \times S^1;\mathbb{Z})$ is a monomorphism and also the direct sum of the two homomorphisms induced canonically by the inclusions. We have the following properties.

\begin{itemize}
\item $H_1(B_{S^1,k}(T^2);\mathbb{Z})$ is isomorphic to $\mathbb{Z} \oplus \mathbb{Z} \oplus \mathbb{Z}/|k|\mathbb{Z}$. 
\item The first and the second direct summands of the group $H_1(B_{S^1,k}(T^2);\mathbb{Z}) \cong \mathbb{Z} \oplus \mathbb{Z} \oplus \mathbb{Z}/|k|\mathbb{Z}$ are regarded as the subgroups generated by elements enjoying the following properties.
\begin{itemize}
\item These two elements are represented by the images of sections of the restrictions of the bundle ${T^2}_o \times S^1$ to circles ${S^1}_1$ and ${S^1}_2$ in the interior of ${T^2}_o$ by which elements of some natural basis of the subgroup  of $H_1(B_{S^1,k}(T^2);\mathbb{Z}) \cong \mathbb{Z} \oplus \mathbb{Z} \oplus \mathbb{Z}/|k|\mathbb{Z}$ generated by all elements whose orders are infinite are represented. 
\item The two elements of $H_1({T^2}_o;\mathbb{Z}) \cong \mathbb{Z} \oplus \mathbb{Z}$ form the basis of course.
\end{itemize}
\item The third summand of the group is generated by an element represented by a fiber of the circle bundle.
\end{itemize}	

We investigate the rational and integral cohomology rings of $B_{S^1,k}(T^2)$ by Poincar\'e duality theorem for it or intersection theory. By universal coefficient theorem, $H^1(B_{S^1,k}(T^2);\mathbb{Z})$ is isomorphic to $\mathbb{Z} \oplus \mathbb{Z}$. Poincar\'e duality theorem shows that this is isomorphic to $H_2(B_{S^1,k}(T^2);\mathbb{Z})$. 
This is generated by elements represented by the total spaces of the restrictions of the circle bundle $B_{S^1,k}(T^2)$ to ${S^1}_1$ and ${S^1}_2$. Remember that ${S^1}_i$ is a circle in ${T^2}_o \subset T^2$ and defined before.  
Choose one of these tori. By using some smooth isotopy, we can move it to another place in such a way that the original torus and the new torus are mutually disjoint. For two tori, we may regard that the intersection can be taken as the fiber, by which a generator of the summand $\mathbb{Z}/|k|\mathbb{Z}$ of $H_1(B_{S^1,k}(T^2);\mathbb{Z}) \cong \mathbb{Z} \oplus \mathbb{Z} \oplus \mathbb{Z}/|k|\mathbb{Z}$ is represented. Note that $H^2(B_{S^1,k}(T^2);\mathbb{Z}) \cong \mathbb{Z} \oplus \mathbb{Z} \oplus \mathbb{Z}/|k|\mathbb{Z}$ by Poincar\'e duality theorem. This means that $B_{S^1,k}(T^2)$ enjoys the property on the rational cohomology ring in Theorem \ref{thm:5} and that it does not enjoy the property on the integral cohomology ring in Main Theorem \ref{mthm:1}.

As a desired manifold admitting a directed round fold map into ${\mathbb{R}}^2$, we consider a connected sum of two copies of $S^2 \times S^1$ and the total space of a circle bundle over $S^2$ whose Euler number is $k$.     
\end{proof}
	\begin{Rem}
		We do not know whether the converse of Main Theorem \ref{mthm:1} holds. Problems like this are important of course.
	\end{Rem}

 \begin{Rem}
		Main Theorem \ref{mthm:1} does not hold where the coefficient ring is a finite commutative ring. This is also pointed out in \cite{kitazawasaeki1}.
	\end{Rem}
\section{Explicit fold maps and restrictions on the manifolds, and Main Theorems.}
\begin{Ex}[\cite{kitazawa0.3}]
	According to \cite{milnor1}, followed by \cite{eellskuiper}, $7$-dimensional homotopy spheres are completely and explicitly classified. If we consider the orientations, there exist exactly $28$ types of $7$-dimensional homotopy spheres. Such homotopy spheres of exactly $16$ of $28$ are the total spaces of linear bundles over $S^4$ whose fibers are the unit sphere $S^3$. $7$-dimensional standard spheres are also of one of such types. All $7$-dimensional homotopy spheres are represented as connected sums of two such homotopy spheres.
	
	This means that $7$-dimensional homotopy spheres admit directed round fold maps into ${\mathbb{R}}^4$ from Theorem \ref{thm:2}. It is well-known that if in the case where the singular set is connected, then it is a standard sphere. This is due to theory of so-called {\it special generic} maps in \cite{saeki2}, some of which will be discussed later. Homotopy spheres of exactly $16$ of $28$, which are the total spaces of linear bundles over $S^4$ before, admit directed round fold maps into ${\mathbb{R}}^4$ whose singular sets consist of exactly two connected components. Furthermore, the converse holds by Theorem \ref{thm:1} and Proposition \ref{prop:3}. Every $7$-dimensional homotopy sphere admits a directed round fold map into ${\mathbb{R}}^4$ whose singular set consists of exactly three connected components.
	This means that difference of the differentiable structures of homotopy spheres and that of topological types of round fold maps of an explicit class are closely related. Such facts have been already discovered in special generic maps, which is presented here shortly. See also \cite{saekisakuma1, saekisakuma2} for example. 
\end{Ex}
A fold map is said to be a {\it special generic} map if the indices of singular points of it are always $0$. Morse functions with exactly two singular points on homotopy spheres and canonical projections of unit spheres are simplest examples. 
In the case where the dimensions of the spaces of the targets are not sufficiently high or at most $4$ (with the conditions forcing the fundamental groups to be trivial or free groups for example), manifolds admitting such maps in such situations are always represented as connected sums of the total spaces of smooth bundles whose fibers are homotopy spheres (in considerable cases). As a recent new study, the author has been studying cases where the dimensions of the spaces of the targets are greater than $4$ with the conditions forcing the fundamental groups to be trivial or in the simple-connected cases. We present an example as a theorem. We also review our proofs in related preprints here. We also review some of fundamental arguments in \cite{saeki2}.
\begin{Thm}[\cite{kitazawa2} etc.]
	\label{thm:6}
	There exists a pair $(M_1,M_2)$ of $9$-dimensional closed and simply-connected manifolds enjoying the following properties.
	\begin{enumerate}
		\item \label{thm:6.1}
		 For $i=1,2$, $H_2(M_i;\mathbb{Z})$ is isomorphic to $\mathbb{Z}/2\mathbb{Z} \oplus \mathbb{Z}/2\mathbb{Z}$, $H_3(M_i;\mathbb{Z})$ is the trivial group, and $H_4(M_i;\mathbb{Z})$ is isomorphic to $\mathbb{Z}$.
		\item \label{thm:6.2} 
		For $M_1$ and $M_2$, consider the subgroups of their integral cohomology groups generated by all elements whose orders are infinite. They have the structures of subrings of the integral cohomology rings and they are isomorphic to the integral cohomology ring of $S^4 \times S^5$.
		\item \label{thm:6.3}
		 The integral cohomology rings of $M_1$ and $M_2$ are not isomorphic.
		\item \label{thm:6.4}
		 $M_1$ admits a special generic map into ${\mathbb{R}}^n$ if and only if $n=5,6,7,8,9$.
		\item \label{thm:6.5}
		 $M_2$ admits a special generic map into ${\mathbb{R}}^n$ if and only if $n=6,7,8,9$.
		\end{enumerate}
\end{Thm}
This is for a kind of appendices to Main Theorems. We review our proof of this theorem according to \cite{kitazawa2} in a way a bit different from the original one. 
\begin{proof}
[A proof of Theorem \ref{thm:6}]
	We can construct a closed and simply-connected manifold $M_1$ and a special generic map $f_1:M_1 \rightarrow {\mathbb{R}}^5$ and we explain about the construction. Note that in \cite{kitazawa2}, we have given another construction.
	
	We have a $5$-dimensional compact and simply-connected manifold $P$ smoothly immersed into ${\mathbb{R}}^5$ such that $H_2(P;\mathbb{Z})$ is isomorphic to $\mathbb{Z}/2\mathbb{Z} \oplus {\mathbb{Z}}{2\mathbb{Z}}$, that $H_3(P;\mathbb{Z})$ is the trivial group and that has the (simple) homotopy type of a $3$-dimensional polyhedron. This is regarded as a result due to fundamental arguments on differential topology. More explicitly, we can also have this from complete classifications of $5$-dimensional closed and simply-conencted manifolds in the topology category, the PL category or equivalently the piecewise smooth category, and the smooth category, presented in \cite{barden}.
	They are equivalent in all these categories. We have a desired manifold by removing a copy of the smoothly embedded copy of the $5$-dimensional unit disk from a certain manifold in the paper. This $5$-dimensional closed and simply-connected manifold is used later, in the construction of $M_2$ as $M^{\prime}$.
	 
	We can construct the product map of a Morse function with exactly one singular point on a copy of the unit disk $D^{5}$ obtained by considering a natural height and the identity map on $\partial P$. We can construct this as a map onto a small collar neighborhood $N(\partial P)$ of $\partial P$. In the complementary set $P-{\rm Int}\ N(\partial P)$ of its interior in $P$, we can construct a trivial smooth bundle over the set whose fiber is a $4$-dimensional standard sphere. We can glue them naturally to obtain a smooth surjection onto $P$. By composing the immersion, we have a desired special generic map $f_1:M_1 \rightarrow {\mathbb{R}}^5$ on a suitable closed and connected manifold $M_1$.
	By considering some propositions on fundamental groups and homology groups in section 3 of \cite{saeki2}, we can see that $M_1$ is a simply-connected manifold enjoying (\ref{thm:6.1}) and (\ref{thm:6.2}). 
	
	We explain about the non-existence of special generic maps into ${\mathbb{R}}^n$ for $n=1,2,3,4$. According to the presented theory, if such a map exists, then this is represented as the composition of a surjection onto an $n$-dimensional compact and simply-connected manifold $P_n$ with some smooth immersion into ${\mathbb{R}}^n$. Furthermore, \cite{nishioka} shows that the integral homology group of the $n$-dimensional compact and simply-connected manifold is free. Moreover theory of \cite{saeki2} before states that $H_j(M_1:\mathbb{Z})$ is isomorphic to $H_j(P_n:\mathbb{Z})$ for $1 \leq j \leq 9-n$. This is a contradiction.
	
		We explain about the existence of special generic maps into ${\mathbb{R}}^n$ for $n=6,7,8,9$ by applying theory first discovered in the preprint \cite{kitazawa4} of the author. In the construction of the special generic map $f_1:M_1 \rightarrow {\mathbb{R}}^5$. The preesnted product map of a Morse function and the identity map on $\partial P$ or a connected component ${\partial}_0 N(\partial P):=\partial N(\partial P) \bigcap {\rm Int}\ P$ of the boundary of the collar neighborhood $N(\partial P)$, and the projection onto the complementary set $P-{\rm Int}\ N(\partial P)$, are glued. We can glue them by the product map of two diffeomorphisms. 
		
		We present the product map of the diffeomorphisms more precisely. First we give suitable identifications of fibers of the two trivial bundles over the boundary ${\partial}_0 N(\partial P)$ and the complementary set $P-{\rm Int}\ N(\partial P)$ in $P$ where the fibers are the unit disk $D^{10-n}=D^5$ and the unit sphere $\partial D^{5}=S^{4}$, respectively. The diffeomorphisms are the identification between the base spaces and the diffeomorphism on the fiber $S^{4}$, regarded as the identity map on this fiber.
		
		Our function defined from a natrual height on the unit disk $D^{5}$ and used here as the Morse function for the product map is regarded as a restriction of a canonical projection of the unit sphere $S^4$ into $\mathbb{R}$. More precisely, this is restricted to a hemisphere. We do not present the rigorous definition of a canonical projection of a unit sphere and its hemisphere. However, we can define naturally and this is defined rigorously in related preprints by the author for example. As a fundamental property, our height function here is represented as the composition of the following two maps.
		\begin{itemize}
			\item The restriction to the hemisphere of a canonical projection of the unit sphere $S^4$ to ${\mathbb{R}}^{n-4}$ in the case $n=6,7,8$ and that to the hemisphere of the canonically defined smooth embedding of the unit sphere $S^4$ to ${\mathbb{R}}^{n-4}={\mathbb{R}}^{5}$ for $n=9$.
			\item A canonical projection of ${\mathbb{R}}^{n-4}$ to $\mathbb{R}$. 
			\end{itemize}
		Although we do not define canonical projections of Euclidean spaces here rigorously, we may give definitions in the canonical way.
		
		Here we construct special generic map into ${\mathbb{R}}^n$ for $n=6,7,8,9$. Around $\partial P$ or ${\partial}_0 N(\partial P)$, we replace the existing product map by the new product map of the following two.
			\begin{itemize}
				\item The restriction to the hemisphere of a canonical projection of the unit sphere $S^4$ to ${\mathbb{R}}^{n-4}$ in the case $n=6,7,8$ and that to the hemisphere of the canonically defined smooth embedding of the unit sphere $S^4$ to ${\mathbb{R}}^{n-4}={\mathbb{R}}^{5}$ for $n=9$. This is presented just before. Furthermore, the space of the target is suitably restricted to a half-space of the Euclidean space.
				\item The identity map on ${\partial}_0 N(\partial P)$.
				\end{itemize}
			
			We replace the projection of the trivial bundle over $P-{\rm Int}\ N(\partial P)$ by the product map of the following two maps. 
		
			\begin{itemize}
			\item A canonical projection of the unit sphere $S^{9-5}=S^4$ to ${\mathbb{R}}^{n-5}$.
			\item The identity map on $P-{\rm Int}\ N(\partial P)$.
			\end{itemize}
		
		By gluing the two maps by the product of the diffeomorphisms before, we have a special generic map into ${\mathbb{R}}^n$ instead. We can also construct this in such a way that the composition of this with a canonical projection to ${\mathbb{R}}^5$ is the originally constructed special generic map $f_1:M_1 \rightarrow {\mathbb{R}}^5$. This completes our exposition on the property (\ref{thm:6.4}). 
		
		According to \cite{barden} for example, we have a $5$-dimensional closed and simply-connected manifold $M^{\prime}$ whose 2nd integral homology group is isomorphic to $\mathbb{Z}/2\mathbb{Z} \oplus \mathbb{Z}/2\mathbb{Z}$. This is also a rational homology sphere or a closed manifold whose rational homology group is isomorphic to that of a sphere. We have $M^{\prime}$ as one we can smoothly immerse and embed into ${\mathbb{R}}^6$. Remember that this manifold is also presented in obtaining the $5$-dimensional manifold $P$ before.
		
		We put $M_2:=M^{\prime} \times S^{4}$. 
	 We explain about the non-existence of special generic maps on $M_2$ into ${\mathbb{R}}^n$ for $n=1,2,3,4,5$. In the case $n=1,2,3,4$, we can argue as in the case of $M_1$. 
	 
	We explain about the case $n=5$.
	 
	We apply some arguments in the third section of \cite{saeki2}, which are also regarded as essential in main ingredients of \cite{kitazawa1}.
	
	Assume that there exists a special generic map $f_{2,5}:M_2 \rightarrow {\mathbb{R}}^5$, then by \cite{saeki2}, there exists a $5$-dimensional compact and simply-connected manifold $P_5$ smoothly immersed into ${\mathbb{R}}^5$ and $f_2$ is represented as the composition of a surjection onto $P_5$ with the smooth immersion. Furthermore, $H_j(M_2;A)$ is isomorophic to $H_j(P_5;A)$ for any commutative ring $A$ and $j=1,2,3,4$. Moreover these isomorphisms are induced by the surjection, denoted by $q_{f_{2,5}}:M_2 \rightarrow P_5$. We put $A:=Z/2\mathbb{Z}$ and we can have
	the cup product for the ordered pair of some element of $H^2(M_2;A)$ and some element of $H^3(M_2;A)$, which is not the zero element. We can take such elements by the definition that $M_2$ is the product of $M^{\prime}$ and $S^4$ and K\"unneth theorem. This is also the pull-back of the cup product of the ordered pair of some element of $H^2(P_5;A)$ and some element of $H^3(P_5;A)$ for the map $q_{f_{2,5}}$. However $P_5$ has the {simple} homotopy type of $4$-dimesional polyhedron. This is a contradiction. We can also see that the properties (\ref{thm:6.1}), (\ref{thm:6.2}) and (\ref{thm:6.3}) are enjoyed by considering the integral cohomology rings and K\"unneth theorem.

We consider the product map of the canonical projection of the unit sphere $S^{9-5}=S^4$ into ${\mathbb{R}}^{n^{\prime}}$ with $n^{\prime}=1,2,3,4$
and the identity map on $M^{\prime}$. We can regard this as a special generic map $f_{2,n^{\prime}+5}:M_2 \rightarrow {\mathbb{R}}^{n^{\prime}+5}$ for $n^{\prime}=1,2,3,4$. We can see $M_2$ admits a special generic map into ${\mathbb{R}}^n$ for $n=6,7,8,9$. For this, remember that $M^{\prime}$ is smoothly immersed and embedded into ${\mathbb{R}}^6$. 

This argument is also presented in a more general manner in \cite{kitazawa1}. This completes the proof of the property (\ref{thm:6.5}). 

This completes the proof.		      	
	\end{proof}
In short, difference of the cohomology rings in the case where the coefficient ring is $\mathbb{Z}$ affects difference in the existence of special generic maps into Euclidean spaces and their dimensions. In addition, these two manifolds cannot be distinguished by their fundamental groups, their integral homology groups and the subgroups of their integral cohomology rings generated by all elements whose orders are infinite and the structures of the subrings of the integral cohomology rings we can induce there. 
We present a very explicit case here and this with its proof is presented in \cite{kitazawa2} in a more general manner. As presented there, we can also have cases of $8$-dimensional closed and simply-connected manifolds for example.

For related studies, see also \cite{kitazawa3} for example.

Our Main Theorems have pointed out similar difference affects difference in types of round fold maps into ${\mathbb{R}}^2$. Note that we cannot consider about fundamental groups in our new cases. It is well-known that $3$-dimensional closed and orientable manifolds are determined by their fundamental groups in considerable cases.

\section{Acknowledgement}
The author would like to thank Osamu Saeki and Takahiro Yamamoto for related rigorous discussions on special generic maps, round fold maps, Main Theorems and Theorem \ref{thm:6}, a previous result of us in \cite{kitazawa2}. The author also would like to thank Masaharu Ishikawa and Yuya Koda for positive and interesting comments on our present study and our related work \cite{kitazawasaeki1}.

	\end{document}